\documentclass[a4paper,11pt]{article}
\usepackage{graphicx} 

\usepackage{latexsym,amsmath,enumerate,amssymb,amsbsy,amsthm,lscape, hyperref} 
\usepackage {graphicx}
\usepackage{float,color,tikz}
\usepackage{subcaption}

\usepackage{algorithm}
\usepackage{algpseudocode}

\theoremstyle{plain}
\newtheorem{theorem}{Theorem}[section]
\newtheorem{proposition}[theorem]{Proposition}
\newtheorem{lemma}[theorem]{Lemma}
\newtheorem{corollary}[theorem]{Corollary}

\theoremstyle{definition}

\theoremstyle{remark}

\newtheorem{example}{Example}[section]

\title{Good Integers: A Concise Completion of the Non-Coprime Case}
\author{Somphong Jitman\footnote{S. Jitman is with the Department of Mathematics, Faculty of Science, Silpakorn University, Nakhon Pathom 73000, THAILAND}}
\date{October 17, 2025}

\begin{document}

\maketitle

\begin{abstract}
   For coprime nonzero integers $a$ and $b$, a positive  integer $\ell$ is said to be {\em good} with respect to $a$ and $b$ if there exists a positive integer $k$ such that  $\ell |(a^{k}+b^{k})$. Since the early 1990s, such  classical good integers have been studied intensively for their   number–theoretic structures and for applications, notably in coding theory. 
This work completes the study  by relaxing the coprimality hypothesis and treating the non-coprime case $\gcd(a,b)\neq1$ in a concise and  self-contained way.  The results are presented  in terms of   the classical coprime criterion   and $p$-adic valuations of $\ell$. As a consequence, whenever  $\ell$ is good, all admissible exponents form a single arithmetic progression with an explicit starting point and period. Some special cases are discussed  in  the non-coprime setting. A practical decision procedure is developed that decides the goodness of a given integer and explicitly enumerates the full set of admissible exponents. Several    illustrative examples   are presented.

\end{abstract}

 \section{Introduction}
The notion of \emph{good integers} was formalized by Moree \cite{M1997} in the study of divisors of $a^k+b^k$.  
For fixed coprime nonzero integers $a$ and $b$, a positive  integer $\ell$ is said to be {\em good} with respect to $a$ and $b$ if there exists a positive integer $k$ such that  $\ell |(a^{k}+b^{k})$; otherwise $\ell$ is said to be \emph{bad}.
Denote by  $G_{(a,b)}$  the set of all good integers with  respect to $a$ and $b$. For each $\ell\in G_{(a,b)}$, a positive integer
$k$ with $\ell|(a^{k}+b^{k})$ is called an \emph{admissible exponent} of $\ell$ with respect
to $a$ and $b$, and we denote by
$
  \mathcal{K}_{(a,b)}(\ell)\ :=\ \{ k\in\mathbb{N}:\ \ell |(a^{k}+b^{k}) \}
$
the set of such  admissible exponents.
The odd good integers were analyzed in \cite{M1997}.  Prior to that formal treatment, aspects of $G_{(q,1)}$ for prime powers $q$ appeared in work related to  Bose–Chaudhuri–Hocquenghem (BCH) codes in \cite{KG1969}. 
Subsequent developments connected $G_{(q,1)}$ to average hull dimensions of cyclic codes \cite{S2003}, to the enumeration of Euclidean self-dual cyclic codes \cite{JLX2011}, to the structure and count of self-dual abelian codes via group algebras \cite{JLLX2012}, and to hull computations for cyclic/negacyclic codes of arbitrary lengths \cite{SJLP2015}. 
More recently, attention turned to \ {even} good integers and finer subclasses oddly-good and evenly-good integers in \cite{JPR2020, J2018a}. Through these studies, a complete characterization has been established, along with various applications in coding theory.

As noted above, the existing literature addresses good integers only in the coprime setting, i.e., under the classical assumption $\gcd(a,b)=1$.
 This paper  completes the picture by relaxing the coprimality hypothesis and giving a concise and self-contained treatment of the non-coprime case. 
Precisely, for nonzero integers $A$ and $B$, a positive integer $L$ is said to be \emph{good} with respect to $A$ and $B$ if there exists a positive integer $K$ such that $L|(A^{K}+B^{K})$. Denote by $G_{(A,B)}$ the set of all good integers defined with respect to $A$ and $B$. For each $L\in G_{(A,B)}$, a positive integer $K$ with $L|(A^{K}+B^{K})$ is called an \emph{admissible exponent} of $L$ (with respect to $A$ and $B$), and we define
\[
  \mathcal{K}_{(A,B)}(L)\ :=\ \bigl\{K\in\mathbb{N}:\ L|(A^{K}+B^{K})\bigr\}.
\]
The goal of this paper is to give a complete characterization of good integers in the non-coprime setting.  In addition, we determine the structure of admissible exponents, present a practical decision algorithm, and include illustrative computations.

The paper is organized as follows.  In Section~\ref{sec:2}, notations and   the core results in the classical coprime case are recalled together with the analysis of  the structure of admissible exponents. The non-coprime case is developed in Section~\ref{sec:3}, including a full characterization, the structure of admissible exponents, several special cases, an explicit decision algorithm, and illustrative examples.  Finally, summary and discussion are given in Section \ref{sec:4}.

\section{Preliminaries} \label{sec:2}
In this section, we set the notation used throughout the paper and recall necessary known results for the classical coprime setting together with the investigation of the  admissible exponents.  We begin with basic divisibility conventions, $p$-adic valuations, and the decomposition of an integer relative to a prescribed set of primes. We then summarize the core criteria for good integers. Finally,  a behavior  analysis  of  the admissible exponents  is established.

\subsection{Notations}
For integers $m$ and $n\neq 0$, we say that \emph{$m$ is divisible by $n$} if there exists an integer $t$ such that $m=nt$; in this case we write $n| m$, otherwise $n\nmid m$. 
For a prime $p$ and an integer $n\neq 0$, the \emph{$p$-adic valuation} $\nu_p(n)$ is the largest integer $e $ such that $p^e| n$ and $p^{e+1}\nmid n$.  Equivalently, $e$ is the exponent of $p$ in the prime factorization of $n$. 
For coprime integers $a$ and $n\geq 1$,   let ${\rm ord}_n(a)$ be the multiplicative order of $a$ modulo $n$. In addition, if $b$ is an integer such that  $\gcd(b,n)=1$, the element $b$ has the inverse $b^{-1}$ modulo $n 
$ and $
   {\rm ord}_n\big(ab^{-1}\big) 
$ refers to  the multiplicative order of  the element $ab^{-1}$ modulo $n$.

For a finite set  $S$  of primes and positive integer $n$, let 
\[
  \lambda_S(n) := \prod_{p\in S} p^{\nu_p(n)} ~ \text{ and } ~
  \lambda_{S'}(n)  := \prod_{p\notin S} p^{\nu_p(n)}.
\]
Then we have the following decomposition of $n$: 
\begin{align} \label{eq:decom}
    n = \lambda_S(n) \lambda_{S'}(n) 
\end{align}  with $\gcd(\lambda_S(n) , \lambda_{S'}(n))=1$. The $\lambda_S(n)$ part is regarded as $1$ if $S=\emptyset$.   Consequently,
$\lambda_{S'}(n)=n$ and the factorization \eqref{eq:decom} becomes
$n=\lambda_S(n)\lambda_{S'}(n)=1\cdot n$.

This factorization is key for analyzing good integers in the non-coprime setting in Lemma~\ref{lem:equiv-three-gamma}  and  Theorem~\ref{thm:equiv-G-gamma}. Illustrative computations are presented in the following examples.
 \begin{example}
Let $n=360=2^{3}\cdot 3^{2}\cdot 5$ and $S=\{2,5\}$.
Then
\[
\lambda_S(n)=2^{\nu_2(n)}\cdot 5^{\nu_5(n)}=2^{3}\cdot 5=40\] and 
\[
\lambda_{S'}(n)=\prod_{p\notin S}p^{\nu_p(n)}=3^{\nu_3(n)}=3^{2}=9.
\]
Hence, $n=\lambda_S(n)\lambda_{S'}(n)=40\cdot 9$ and 
$\gcd(40,9)=1$.
\end{example}

\begin{example}
Let $n=24,016=2^{4}\cdot 3\cdot 7^{2}\cdot 11$ and $S=\{3,7,13\}$.
Here, $13\nmid n$ which implies that  $\nu_{13}(n)=0$  and  $13^{0}=1$.
Then
\[
\lambda_S(n)=3^{\nu_3(n)}\cdot 7^{\nu_7(n)}\cdot 13^{\nu_{13}(n)}
=3^{1}\cdot 7^{2}\cdot 13^{0}=147
\]
and 
\[
\lambda_{S'}(n)=\prod_{p\notin S}p^{\nu_p(n)}=2^{\nu_2(n)}\cdot 11^{\nu_{11}(n)}
=2^{4}\cdot 11=176.
\]
Hence,  $n=\lambda_S(n)\lambda_{S'}(n)=147\cdot 176$ and $\gcd(147,176)=1$.
\end{example}

\subsection{Classical Coprime Case}
In this subsection,    the core theory for good integers in the classical setting
$\gcd(a,b)=1$ is summarized. In this case, a positive integer $\ell$ is said to be \emph{good} with respect to $a$ and $b$ if there exists  a positive integer $k$ such that $\ell |(a^{k} + b^{k})$.
 These results form the baseline to which the non-coprime case will be
deduced in the next section.

We begin with  necessary conditions that any positive integer must satisfy to  be good in the classical coprime setting; these will be used repeatedly.

\begin{lemma}[{\cite[Lemma 2.1]{J2018a}}]
	\label{gcd-1-ab} Let $a$ and $ b$  be nonzero coprime  integers and let $d$ be a positive    integer. If    $d\in G_{(a,b)}$, then     $\gcd(a,d)=1=\gcd(b,d)$.  
\end{lemma}

\subsubsection{Good Odd Integers}
The following results on good odd integers were established  in  \cite{M1997}, providing a foundational characterization and motivating much subsequent work on the classical theory of good integers as well as applications in various fields. 
\begin{proposition}[{\cite[Proposition 2]{M1997}}]\label{2order}
Let $p$ be an odd prime and $r\ge 1$. If $p^r \in G_{(a,b)}$, then ${\rm ord}_{p^r}(ab^{-1})=2s$, where $s$ is the least positive integer such that $(a b^{-1})^s\equiv -1 \pmod{p^{r}}$.
\end{proposition}

\begin{proposition}[{\cite[Proposition 4]{M1997}}]\label{ord}
Let $p$ be an odd prime and $r\ge 1$. Then ${\rm ord}_{p^r}(ab^{-1})={\rm ord}_{p}(ab^{-1}) p^i$ for some $i\ge 0$.
\end{proposition}

\begin{theorem}[{\cite[Theorem 1]{M1997}}]\label{goodP}
Let $d>1$ be odd. Then $d\in G_{(a,b)}$ if and only if there exists $s\ge 1$ such that
$2^s\Vert {\rm ord}_{p}(ab^{-1})$ for every prime $p|d$.
\end{theorem}

\subsubsection{Good Even Integers}
The characterization of good odd integers in \cite{M1997} has been broadened to encompass good even integers in \cite{JPR2020,J2018a}. For completeness,  the essential criteria and structural consequences of  the even case are summarized below.

\begin{proposition}[{\cite[Proposition 2.3]{JPR2020}}] \label{evengood}
Let $a,b$ be coprime odd integers and $\beta\ge 1$. The following are equivalent.
\begin{enumerate}[$(i)$]
\item $2^\beta\in G_{(a,b)}$.
\item $2^\beta|(a+b)$.
\item $a b^{-1} \equiv -1 \pmod{2^\beta}$.
\end{enumerate}
\end{proposition}

\begin{proposition}[{\cite[Proposition 2.2]{J2018a}}] \label{prop2d}
Let $a$, $b$, and $d>1$ be pairwise coprime odd integers. Then $d\in G_{(a,b)}$ if and only if $2d\in G_{(a,b)}$. In this case, ${\rm ord}_{2d}(ab^{-1})={\rm ord}_{d}(ab^{-1})$ is even.
\end{proposition}

\begin{proposition}[{\cite[Proposition 2.7]{JPR2020}}]  \label{prop2m}
Let $a,b,d>1$ be pairwise coprime odd integers and $\beta\ge 2$. Then $2^\beta d\in G_{(a,b)}$ if and only if $2^\beta|(a+b)$ and $d\in G_{(a,b)}$ with $2\Vert {\rm ord}_{d}(ab^{-1})$. In this case, ${\rm ord}_{2^\beta}(ab^{-1})=2$ and $2\Vert {\rm ord}_{2^\beta d}(ab^{-1})$.
\end{proposition}

\subsection{Behavior of Admissible Exponents}
For fixed coprime integers  $a, b$ and  a positive integer $\ell\in G_{(a,b)}$,   the admissible exponents of $\ell$ are presented.  In particular, for $\ell\notin\{1,2\}$ all admissible exponents have the same parity. Later, these behaviors    will  be discussed  in the non-coprime case    in Subsection \ref{ssec:admis}

\begin{lemma}\label{lem:k}
Let $a$ and $b$ be coprime  integers and let $\ell$ and $k$ be  positive integers. Then   $\ell|(a^k+b^k)$
  if and only if  $\ell|(a^{\alpha k}+b^{\alpha k})$ for all odd positive integers $\alpha$. 
\end{lemma}
\begin{proof}  First, assume that  $\ell|(a^k+b^k)$ and let $\alpha $ be an odd positive integer.  Then $\ell\in G_{(a,b)}$. By Lemma \ref{gcd-1-ab}, we have $\gcd(a,\ell)=1 $ and $\gcd(b,\ell)=1$ which implies that  $b$ is invertible modulo $\ell$. Hence, 
$ (a b^{-1})^k \equiv -1 \pmod{\ell}$.
Since $\alpha$ is odd, we have  $(a b^{-1})^{\alpha k}\equiv ((a b^{-1})^k)^{\alpha}\equiv (-1)^{\alpha} \equiv -1 \pmod{\ell}$, so $(a b^{-1})^{\alpha k}\equiv -1\pmod{\ell}$.  Consequently,   $\ell|(a^{\alpha k}+b^{\alpha k})$ for all odd positive integers $\alpha$.

The converse follows immediately by taking  an odd positive integer $\alpha=1$.
\end{proof}

We observe that  $\mathcal K_{(a,b)}(1)=\mathbb{N}$ for all coprime integers $a$ and $b$,  and   $\mathcal K_{(a,b)}(2)=\mathbb{N}$   for all coprime odd integers $a$ and $b$.  For $\ell>2$, we have the following results for the admissible exponents.

\begin{lemma}\label{lem:min-k}
Let $a$ and $b$ be coprime positive integers and let $\ell\in G_{(a,b)}\setminus \{ 1,2\}$.
Then   
$
  {\rm ord}_{\ell} \left(ab^{-1}\right)  $  is even and  
 \[\min ( \mathcal K_{(a,b)}(\ell))=\frac{{\rm ord}_{\ell} \left(ab^{-1}\right) }{2}.\]
Moreover,  
\[
 \mathcal K_{(a,b)}(\ell)=\Bigl\{\alpha\cdot\frac{{\rm ord}_{\ell} \left(ab^{-1}\right) }{2}\ :\ \alpha\ \text{ is an odd positive integer}\Bigr\}\] is infinite.
\end{lemma}

\begin{proof} Let $T=\Bigl\{\alpha\cdot\frac{{\rm ord}_{\ell} \left(ab^{-1}\right) }{2}\ :\ \alpha\ \text{ is an odd positive integer}\Bigr\}$. 
From  Theorem~\ref{goodP}  and Propositions~\ref{evengood}--\ref{prop2m},  it follows that  $
  {\rm ord}_{\ell} \left(ab^{-1}\right)  $  is even and    \[(a b^{-1})^{{\rm ord}_{\ell}(ab^{-1})/2}\equiv -1\pmod{\ell}.\]
Hence, 
\[
  \ell  \bigm| \left(a^{{\rm ord}_{\ell}(ab^{-1})/2}+b^{{\rm ord}_{\ell}(ab^{-1})/2}\right).
\]
We note that   $-1\in\langle   ab^{-1}\rangle$ as  a cyclic multiplicative group  contained a unique element of order $2$,
namely $\left(ab^{-1}\right)^{\frac{{\rm ord}_{\ell} \left(ab^{-1}\right) }{2}} \equiv -1 \pmod \ell$.  Let $k\in \mathcal K_{(a,b)}(\ell)$. Then   $(ab^{-1})^{k}\equiv -1 \pmod \ell$ which implies that  ${\frac{{\rm ord}_{\ell} \left(ab^{-1}\right) }{2}}|k$.  Hence, \[\min ( \mathcal K_{(a,b)}(\ell))=\frac{{\rm ord}_{\ell} \left(ab^{-1}\right) }{2}\] and $\mathcal K_{(a,b)}(\ell)\subseteq T$.   By Lemma \ref{lem:k},  it follows that $T\subseteq \mathcal K_{(a,b)}(\ell)$.  Hence, $ \mathcal K_{(a,b)}(\ell) =T$  which is clearly infinite.
\end{proof}

\begin{lemma}\label{lem:K-structure}
Let $a$ and $b$ be coprime integers and let $\ell\in G_{(a,b)}$.  If $\ell\notin\{1,2\}$, then all integers  $k\in\mathcal K_{(a,b)}(\ell)$ have the same parity. 
\end{lemma}

\begin{proof} Assume that  $\ell\notin\{1,2\}$.
By Lemma~\ref{lem:min-k}, we have the explicit
 form 
\[
 \mathcal K_{(a,b)}(\ell)=\Bigl\{\alpha\cdot\frac{{\rm ord}_{\ell} \left(ab^{-1}\right) }{2}\ :\ \alpha\ \text{ is an odd positive integer}\Bigr\}.
\]     Let $ k,m \in \mathcal{K}_{(a,b)}(\ell)$.  Then $k=\alpha\cdot\frac{{\rm ord}_{\ell} \left(ab^{-1}\right) }{2}$ and $m=\beta\cdot\frac{{\rm ord}_{\ell} \left(ab^{-1}\right) }{2}$ for some odd positive integers $\alpha$ and $\beta$.  Since $\alpha-\beta$ is even,  it follows that $k-m\equiv  (\alpha-\beta)\cdot\frac{{\rm ord}_{\ell} \left(ab^{-1}\right) }{2} \equiv 0 \pmod 2$  which implies that $k$ and  $m$ have the same parity.
\end{proof}
From Lemmas \ref{lem:min-k} and  \ref{lem:K-structure}, for each $\ell\in G_{(a,b)}\setminus \{ 1,2\}$, the  admissible exponents form  an arithmetic progression with the first term  $\frac{{\rm ord}_{\ell} \left(ab^{-1}\right) }{2}$ and common different  ${{\rm ord}_{\ell} \left(ab^{-1}\right) }$.  Moreover, the parity in Lemma  \ref{lem:K-structure} supports \cite[Proposition 3.1]{J2018a}.

\section{Non-coprime Case} \label{sec:3}
 
In this section,  results for general situation where $A$ and $B$ are not necessarily coprime, is established,  including a full characterization, the structure of admissible exponents, several special cases, an explicit decision algorithm, and illustrative examples.

\subsection{Characterization}

Let $A$ and $B$ be non-zero integers  and set $g=\gcd(A,B)$. Write
\[
A=g a \quad  \text{ and } \quad B=g b,
\]
for some integers $a$ and $b$ such that  $\gcd(a,b)=1$. 
Let $\mathcal{P}(g)$ be the set of prime divisors of $g$. Clearly,  $L=1\in G_{(A,B)}$.    For each   integer $L\geq 2$, based on  the
  decomposition in \eqref{eq:decom} with $S=\mathcal{P}(g)$, it follows that
\[
  L  =  \lambda_{\mathcal{P}(g)}(L) \ell,
\]
where  the $g$-part
 $ \lambda_{\mathcal{P}(g)}(L)=\prod\limits_{p|g}p^{\nu_p(L)} $ and  the $\ell$-part $
  \ell=\prod\limits_{p\nmid g}p^{\nu_p(L)}$.  We note  that   $\gcd(\ell,g)=1$.  Let
\[
  \gamma(L)\ :=\ \max_{p|g}\Big\lceil \frac{\nu_p(L)}{\nu_p(g)}\Big\rceil,
\]
where  $\lceil  x \rceil$ denotes the smallest integer greater than or equal to $x$.  By convention,   the maximum over the empty set is $0$   if   $g=1$.

\begin{example} \label{ex3.1}
Let $A=18$ and  $B=12$. Then $g=\gcd(18,12)=6$ with $\mathcal P(g)=\{2,3\}$ and
$A=ga$, $B=gb$ where $a=3$ and $b=2$.
Let $L=1,200=2^4\cdot 3^1\cdot 5^2$. Then
\[\text{the } g-\text{part }
  \lambda_{\mathcal P(g)}(L)=\prod_{p| g}p^{\nu_p(L)}=2^{4}\cdot 3^{1}=48  ~~\text{and  the }  \ell-\text{part }
  \ell=\prod_{p\nmid g}p^{\nu_p(L)}=5^{2}=25,\]  and $ L=\lambda_{\mathcal P(g)}(L)\ell=48\cdot 25$ 
  with  $\gcd(\ell,g)=\gcd(25,6)=1$.
Since $\nu_2(g)=\nu_3(g)=1$, we have
\[
  \gamma(L)=\max_{p| g}\Big\lceil \tfrac{\nu_p(L)}{\nu_p(g)}\Big\rceil
  =\max\big\{\lceil 4/1\rceil,\ \lceil 1/1\rceil\big\}=4.
\]
\end{example}

\begin{example} 
Let $A=8$  and  $B=27$. Then $g=\gcd(8,27)=1 $ and  $\mathcal P(g)=\varnothing$ which implies that 
$A=ga$ and  $B=gb$ with $a=8$, $b=27$,  and $\gcd(a,b)=1$. Then, for any   $L$, the decomposition gives $L=\lambda_{\mathcal P(g)}(L)\ell=1\cdot L$,  where 
\[\text{the }g\text{-part } 
  \lambda_{\mathcal P(g)}(L)=\prod_{p| g}p^{\nu_p(L)}=1 \text{  and the }\ell\text{-part }
  \ell=\prod_{p\nmid g}p^{\nu_p(L)}=L
\]
with  $\gcd(\ell,g)=1$ trivially.
By convention, $
  \gamma(L)=\max_{p| g}\Big\lceil \tfrac{\nu_p(L)}{\nu_p(g)}\Big\rceil=0$.
\end{example}

\begin{lemma}\label{lem:equiv-three-gamma}  Let $A$ and $B$ be non-coprime integers and let $L$ be a positive integer. With notation as above, the following statements  are equivalent:
\begin{enumerate}
\item[(i)] There exists $K\in \mathbb{N}$ such that $L |(A^K+B^K)$.
\item[(ii)] There exists $\kappa\ge \gamma(L)$ such that $\ell |(a^\kappa+b^\kappa)$.
\item[(iii)] There exists $k\in \mathbb{N}$ such that $\ell |(a^k+b^k)$.
\end{enumerate}
\end{lemma}

\begin{proof}
Write $A=ga$, $B=gb$, and decompose
\[
  L = \lambda_{\mathcal{P}(g)}(L)\ell,
  \]
  where \[
  \lambda_{\mathcal{P}(g)}(L)=\prod_{p|g}p^{\nu_p(L)}
  \quad  \text{and} \quad 
  \ell=\prod_{p\nmid g}p^{\nu_p(L)} .
\]
Recall that 
\[
  \gamma(L)\ :=\ \max_{p|g}\Big\lceil \frac{\nu_p(L)}{\nu_p(g)}\Big\rceil.
\] 

To prove \emph{(i)$\Rightarrow$(ii)},  assume $L|(A^K+B^K)$ for some  positive integer $K$. Since
\[
  A^K+B^K = g^K\bigl(a^K+b^K\bigr),
\]
we derive the following   $p$-adic valuation of $p|(A^K+B^K)$. 

\noindent {\bf Case:} $p|g$.  We have 
\[
  \nu_p(A^K+B^K)=\nu_p\bigl(g^K(a^K+b^K)\bigr)\ \ge\ K\nu_p(g).
\]
Since $L|(A^K+B^K)$,  we get
\[
  \nu_p(L)=\nu_p\bigl(\lambda_{\mathcal{P}(g)}(L)\bigr)\ \le\ \nu_p(A^K+B^K)
  \ \le\ K\nu_p(g).
\]
Hence,  $K\ge \nu_p(L)/\nu_p(g)$ for all $p|g$.  Consequently, we have  $K\ge \gamma(L)$.

\noindent {\bf Case:} $p\nmid g$.   In this case, we have $\nu_p(g)=0$ and
\[
  \nu_p(A^K+B^K)=\nu_p\bigl(a^K+b^K\bigr).
\]
Since $L|(A^K+B^K)$,  $\nu_p(\ell)=\nu_p(L)\le \nu_p(a^K+b^K)$ for all   $p\nmid g$, which implies that  
$\ell|(a^K+b^K)$.

From both cases, (ii) follows by setting $\kappa:=K$.

The implication \emph{(ii)$\Rightarrow$(iii)} follows 
immediately  by taking $k=\kappa$.

To proof \emph{(iii)$\Rightarrow$(i)}, assume $\ell|(a^k+b^k)$ for some positive integer $k$. Since  $\gcd(a,b)=1$ and $\gcd(\ell,ab)=1$,  by Lemma~\ref{lem:k}, it follows that 
\[
  \ell|\bigl(a^{\alpha k}+b^{\alpha k}\bigr)\quad\text{for all odd positive integer}\alpha.
\]
Let  $\alpha$  be chosen such  that
$
  \alpha k \ge \gamma(L)$ and $ K:=  \alpha k$.  
Then,  for each  $p|g$, we have
\[
  \nu_p\bigl(g^K\bigr)=K\nu_p(g)\ \ \ge\ \ \gamma(L)\nu_p(g)\ \ \ge\ \ \nu_p(L)
\]
which implies that  $\lambda_{\mathcal{P}(g)}(L)|g^K$. Since $\alpha$ is odd, $\ell|(a^{\alpha k}+b^{\alpha k})=(a^K+b^K)$ by Lemma~\ref{lem:k}.
Since $\gcd(\ell,g)=1$, it can be concluded that 
\[
  L = \lambda_{\mathcal{P}(g)}(L)\ell \ \Big|\ g^K(a^K+b^K) = A^K+B^K
\]
which proves (i). 
\end{proof}

\begin{theorem}\label{thm:equiv-G-gamma}
With notation as above, for any $L\ge 2$ the following are equivalent:
\begin{enumerate}
\item $L\in G_{(A,B)}$.
\item There exists $\kappa\ge \gamma(L)$ with $\ell |(a^\kappa+b^\kappa)$.
\item $\ell\in G_{(a,b)}$.
\end{enumerate}
\end{theorem}

\begin{proof}
The result  follows  immediately from Lemma~\ref{lem:equiv-three-gamma}.
\end{proof}

Following Lemma~\ref{lem:equiv-three-gamma}, any admissible exponent  $k$ with $\ell|(a^{k}+b^{k})$ can be lifted to a global exponent $K$ by choosing an {odd} multiplier $\alpha$ so that $K=\alpha k\ge \gamma(L)$; then $\ell|(a^{K}+b^{K})$ and $\lambda_{\mathcal P(g)}(L)| g^{K}$, hence $L|(A^{K}+B^{K})$. An  illustrative  computation is given in the next example.

\begin{example} Let  $A=18$, $B=12$, and  $L=1200$.  From Example \ref{ex3.1},  we have 
  $g=\gcd(18,12)=6$,  $A=ga$,  and   $B=gb$,  where $a=3$ and $b=2$.
Then    
$
  L=\lambda_{\mathcal P(g)}(L)\ell=(2^4\cdot 3)\cdot 5^2=48\cdot 25,$ 
  $\ell=25$,  and $\gamma(L)=4$. 
  We note that ${\rm ord} _{\ell}( ab^{-1}) = {\rm ord} _{25}( 3\cdot 2^{-1}) ={\rm ord} _{25}( 14) = 10$ and  $k\equiv \frac{{\rm ord}_{25}(14)}{2}\equiv 5\pmod{10}$ is an admissible exponent  such that
$25|(3^k+2^k)$ which implies that $\ell \in G_{(3,2)}$    (Theorem \ref{goodP} may apply  for general $\ell$).

By the proof of Lemma~\ref{lem:equiv-three-gamma},  let  $K=5$.
Then
\[
  A^5+B^5 = g^5\bigl(a^5+b^5\bigr),
\]
and we already have $25|(a^5+b^5)$. For the $g$-part, we have 
\[
  \nu_2\bigl(g^5\bigr)=5\ge 4=\nu_2\bigl(\lambda_{\mathcal P(g)}(L)\bigr)
  \quad \text{ and }\quad
  \nu_3\bigl(g^5\bigr)=5\ge 1=\nu_3\bigl(\lambda_{\mathcal P(g)}(L)\bigr),
\]
which implies that $\lambda_{\mathcal P(g)}(L)=2^4\cdot 3$ divides $g^5$. Therefore, 
\[
  L=\lambda_{\mathcal P(g)}(L)\ell | g^5(a^5+b^5) = A^5+B^5,
\]
i.e.\ $L\in G_{(A,B)}$.  
\end{example}

\begin{example} \label{ex3.6}
Let $A=18$, $B=12$, and $L =3200 =  2^{7}\cdot 25 $. Then $g=\gcd(18,12)=6$ and we can write $A=ga$, $B=gb$ with $a=3$ and $b=2$.
 Then  $\mathcal P(g)=\{2,3\}$,  
\[
  \lambda_{\mathcal P(g)}(L)=2^{7},\quad \ell=25,\quad  \text{and} \quad
  \gamma(L)=\max \Big\{\Big\lceil\frac{\nu_2(L)}{\nu_2(g)}\Big\rceil,\Big\lceil\frac{\nu_3(L)}{\nu_3(g)}\Big\rceil\Big\}
  =\max\{7,0\}=7.
\]
Since 
$
  {\rm ord}_{25}(ab^{-1})={\rm ord}_{25}(3\cdot 2^{-1})={\rm ord}_{25}(14)=10$, 
$
  k= \frac{{\rm ord}_{25}(14)}{2}=\frac{10}{2}=5$ 
 satisfies $25|(3^{k}+2^{k})$ which implies that  $\ell=25\in G_{(3,2)}$.

By the proof of Lemma~\ref{lem:equiv-three-gamma}, let $K:=3\cdot k=3\cdot 5=15$  is an odd multiple of $k$ and $K\ge \gamma(L)=7$.   
 For the $g$-part,  we have $ \nu_2\bigl(g^{K}\bigr)=K \nu_2(g)=15\cdot 1\ge 7=\nu_2 \bigl(\lambda_{\mathcal P(g)}(L)\bigr)$ and $\nu_3\bigl(g^{K}\bigr)=K \nu_3(g)=15\cdot 1\ge 0=\nu_3 \bigl(\lambda_{\mathcal P(g)}(L)\bigr)$ which implies that  $\lambda_{\mathcal P(g)}(L)=2^{7}| g^{K}$.  Consequently, 
\[
  L=\lambda_{\mathcal P(g)}(L) \ell \bigm| g^{K} (a^{K}+b^{K})  =  A^{K}+B^{K}
\]
which means  $L\in G_{(A,B)}$.
\end{example}

\subsection{Behavior of Admissible Exponents: Non-Coprime Case}

\label{ssec:admis}
In this subsection,  we focus on properties and behavior of   admissible exponents of good integers in non-coprime case. 

 We recall  that an {admissible exponent} of  $ L\in G_{(A,B)}$  is a positive integer $K$ such that  $L|(A^{K}+B^{K})$ and the set   of  admissible exponents of $L$ is denoted by 
\[
  \mathcal{K}_{(A,B)}(L) := \bigl\{K\in\mathbb{N}:\ L|(A^{K}+B^{K})\bigr\}.
\]
In the next proposition, a precise arithmetic progression description of
$\mathcal{K}_{(A,B)}(L)$ is given in terms of   $\ell$-part and    $\gamma(L)$.

\begin{proposition} \label{prop:AP-noncoprime}
Let $A=ga$ and  $B=gb$ with $\gcd(a,b)=1$ and let $L=\lambda_{\mathcal{P}(g)}(L)\ell \in G_{(A,B)}$. Let   \[ L_0= \begin{cases}
  1 &\text{ if } 1\leq \ell \leq 2,\\
     {\rm ord}_{\ell}(ab^{-1})  &\text{ if } \ell \geq 3,
\end{cases} \quad \text{and} \quad r=\begin{cases}
    0 &\text{ if } 1\leq \ell \leq 2,\\
    \frac{L_0}{2} &\text{ if } \ell \geq 3.
\end{cases}\] Then the following statements hold. 
\begin{enumerate} 
 
\item 
$
  \mathcal{K}_{(A,B)}(L)
  =
  \{K\in \mathbb{N}:\ K\equiv r\pmod{L_0}\ \text{ and }\ K\ge \gamma(L)\}
$ is infinite.
\item The minimal admissible exponent is
\[ 
  \min\mathcal{K}_{(A,B)}(L)
  =
  \gamma(L) + \bigl((r-\gamma(L))\bmod L_0\bigr).
\]
\end{enumerate}
\end{proposition}

\begin{proof} Since $L\in G_{(A,B)}$,  we have  $\ell\in G_{(a,b)}$  by  Lemma~\ref{lem:equiv-three-gamma}. The proof is separated into two cases as follows. 

\noindent {\bf Case 1:} $1\leq \ell \leq 2$.      If $\ell=1$, we have  $\ell |(a^K+b^K)$   for all positive integers $K$.  Otherwise, $a$ and $b$ are both odd which implies that 
$2| (a^K+b^K)$ for all positive integers $K$.   By   Lemma~\ref{lem:equiv-three-gamma}, we have  $K\ge\gamma(L)$. Since  $L_0=1$ and $r=0$,  it follows that  
\[
  \mathcal K_{(A,B)}(L)=\{K\ge\gamma(L)\} =
  \{K\in \mathbb{N}:\ K\equiv r\pmod{L_0}\ \text{ and }\ K\ge \gamma(L)\}
\]
and \[
  \min\mathcal{K}_{(A,B)}(L)=\gamma(L) = 
  \gamma(L) + \bigl((r-\gamma(L))\bmod L_0\bigr)\] as desired.

\noindent {\bf Case 2:} $3\leq \ell  $.  Since $\ell\in G_{(a,b)}$, $\gcd(\ell,ab)=1$ by Lemma~\ref{gcd-1-ab}.  Hence,  $L_0={\rm ord}_{\ell}(ab^{-1}) $ is well-defined and it is even by Lemma \ref{lem:min-k}.    By   Lemma~\ref{lem:min-k},  $\ell|(a^{K}+b^{K})$ for all  $K = r+ m L_0 $ and $m\in \mathbb{N}\cup\{0\}$. 
 From the proof of   Lemma~\ref{lem:equiv-three-gamma},  we have   $K\ge \gamma(L)$.
Combining  the conditions, we have 
\[
  \mathcal{K}_{(A,B)}(L)
  =
  \{ K\in \mathbb{N}:\ K\equiv r\pmod{L_0}\ \text{ and }\ K\ge \gamma(L) \}.
\]
 Let $R$ be the unique residue class representative of $r$ modulo $L_0$ in $\{0,1,\dots,L_0-1\}$. By considering 
the least integer $K\ge \gamma(L)$ with $K\equiv R\pmod{L_0}$,  the  least $K$ is
\[
 \min\mathcal{K}_{(A,B)}(L)
  =
  \gamma(L) + \bigl(( R-\gamma(L) )\bmod L_0\bigr).
\]
 
This completes the proof.
\end{proof}

\begin{corollary} \label{prop:parity-noncoprime}
Let $A=ga$, $B=gb$ with $\gcd(a,b)=1$ and $L=\lambda_{\mathcal{P}(g)}(L) \ell\in G_{(A,B)}$.
With $L_0$ and $r$ as in Proposition~\ref{prop:AP-noncoprime}, the following statements hold.
\begin{enumerate}
  \item If $\ell\in\{1,2\}$, then 
        $\mathcal{K}_{(A,B)}(L)=\{K\ge \gamma(L)\}$ with no restriction on parity.
  \item If $\ell\ge 3$, then   every  $K\in\mathcal{K}_{(A,B)}(L)$  has  the same parity as $\frac{L_0}{2}$.
\end{enumerate}
\end{corollary}
\begin{proof}
 The results follow immediately  from Proposition~\ref{prop:AP-noncoprime}.
\end{proof}

\begin{example} 
Let $A=6$ and $ B=3$.  Then $g=\gcd(6,3)=3$, $a=2$ and $b=1$.  Let  $L=15=3\cdot 5$.
Then $\lambda_{\mathcal{P}(g)}(L)=3$, $\ell=5$, and $\gamma(L)=\lceil \nu_3(15)/\nu_3(3)\rceil=1$.
Since $2^2+1^2=5$, we have $\ell\in G_{(a,b)}$ and $L_0={\rm ord}_{5}(ab^{-1})={\rm ord}_{5}(2)=4$, and $r=L_0/2=2$. By Proposition~\ref{prop:AP-noncoprime},
\[
  \mathcal{K}_{(A,B)}(15)=\{K\in \mathbb{N}:\ K\equiv 2\pmod 4 \text{ and } K \geq 1\}
  =\{2,6,10,\dots\},
\]
and $\min\mathcal{K}_{(6,3)}(15)=2$. In particular, all admissible exponents $K$ are even coincided with Proposition~\ref{prop:parity-noncoprime}.
\end{example}

\begin{example} Let  $A=18$, $B=12$, and $L=3200=2^{7}\cdot 25$.  From Example  \ref{ex3.6}, we have $g=\gcd(18,12)=6$, $a=3$,  $b=2$, $
\lambda_{\mathcal P(g)}(L)=2^{7}$,  $\ell=25$,  and 
$\gamma(L) =7$. Consequently, 
\[
L_0={\rm ord}_{25}(ab^{-1})={\rm ord}_{25}(14)=10 \quad\text{and}\quad r=\tfrac{L_0}{2}=5. 
\] By   Lemma~\ref{lem:equiv-three-gamma},  we have $K\ge \gamma(L)=7$. By Proposition~\ref{prop:AP-noncoprime},
the   admissible set is
\[
\mathcal K_{(A,B)}(3200)
=\{ K\in \mathbb{N}:\ K\equiv 5 \pmod{10}\ \text{and}\ K\ge 7 \}
=\{15,25,35,45,\ldots\}.
\]
In particular, $\min\mathcal{K}_{(18,12)}(3200)=15$ and all admissible exponents of $3200$ are odd.  
\end{example}

\subsection{Some Special Cases}

In this subsection, simplified characterizations in some  special cases are presented: pure $g$-part ($\ell=1$), prime power $g=p^s$, square free $g$, and $g$-contained case   ($\nu_p(L)\le\nu_p(g)$ for all $p|g$). In each case, $\gamma(L)$ and the admissible-exponent set admit simple closed forms.

Let $A=g a$, $B=g b$ with $\gcd(a,b)=1$, and let  $L=\lambda_{\mathcal{P}(g)}(L) \ell$. Then we have the following  corollaries. 

\begin{corollary}\label{cor:pure-g-part}
  If $\ell=1$, then
$
  L\in G_{(A,B)} $  if and only if   there exists  an integer $K\ge \gamma(L)$.
In this case, $L\in G_{(A,B)}$ and every $K\ge \gamma(L)$ is admissible.
\end{corollary}

\begin{proof} By setting 
 $\ell=1$,   the results follows from  Lemma~\ref{lem:equiv-three-gamma}  and  Theorem~\ref{thm:equiv-G-gamma}.
\end{proof}

\begin{corollary}If  $g=p^{s}$ with $p$ prime and $s\ge 1$,  then
\[
  L = p^{\alpha} \ell,\quad \alpha=\nu_p(L),\quad \gcd(\ell,p)=1, \quad \text{ and } \quad 
  \gamma(L)=\Big\lceil\frac{\alpha}{s}\Big\rceil.
\]
Moreover, if $\ell\in G_{(a,b)}$, then  $L\in G_{(A,B)}$ and the admissible exponents are exactly
\[
   \mathcal{K}_{(A,B)}(L)=\{ K\in \mathbb{N}:\ K\equiv r\ (\mathrm{mod}\ L_0)\ \text{ and }\ K\ge \lceil \alpha/s\rceil \} 
\]
and 
\[
  \min  \mathcal{K}_{(A,B)}(L)=\Big\lceil\frac{\alpha}{s}\Big\rceil+\bigl((r-\lceil\alpha/s\rceil)\bmod L_0\bigr), 
\] where $L_0$ and $r$ are defined as in Proposition~\ref{prop:AP-noncoprime}.
\end{corollary}

\begin{proof}
Since $g=p^{s}$,  we have $\nu_p(g)=s$  and $\gamma(L)=\max_{p|g}\lceil \nu_p(L)/\nu_p(g)\rceil=\lceil \alpha/s\rceil$. The  rest  follows from  Theorem~\ref{thm:equiv-G-gamma}.
\end{proof}

\begin{corollary} If  $g$ is square free, then 
$
  \gamma(L)=\max_{p|g}\nu_p(L).
$. Moreover, if  $\ell\in G_{(a,b)}$, then $L\in G_{(A,B)}$ and the admissible exponents are exactly
\[
  \mathcal{K}_{(A,B)}(L)=\{ K\in \mathbb{N}:\ K\equiv r\ (\mathrm{mod}\ L_0)\ \text{ and }\ K\ge \max_{p|g}\nu_p(L) \}
\]
where $L_0$ and $r$ are defined as in Proposition~\ref{prop:AP-noncoprime}.
\end{corollary}

\begin{proof}
Since  $g$ is square free,  $\nu_p(g)=1$ for all prime  $p|g$ which implies that  $\gamma(L)=\max_{p|g}\lceil \nu_p(L)/1\rceil=\max_{p|g}\nu_p(L)$. The rest follows from  Theorem~\ref{thm:equiv-G-gamma}.
\end{proof}

\begin{corollary} 
If $\nu_p(L)\le \nu_p(g)$ for  all primes $p|g$, then $\gamma(L)=1$. 
Moreover, if  $\ell\in G_{(a,b)}$,  then $L\in G_{(A,B)}$ and  the admissible exponents are exactly
\[
  \mathcal{K}_{(A,B)}(L)=\{ K\in \mathbb{N}:\ K\equiv r\ (\mathrm{mod}\ L_0)  \}
\]
where $L_0$ and $r$ are defined as in Proposition~\ref{prop:AP-noncoprime}.
\end{corollary}

\begin{proof}
The assumption $\nu_p(L)\le \nu_p(g)$ implies that $\lceil \nu_p(L)/\nu_p(g)\rceil\le 1$ for  all primes  $p|g$. Consequently, $\gamma(L)=1$. The rest follows from  Theorem~\ref{thm:equiv-G-gamma}.
\end{proof}

\subsection{Algorithm: Checking $L\in G_{(A,B)}$}
\label{subsec:algo}
 
We present an effective procedure to decide whether a given modulus $L$ is good with respect to nonzero (possibly non-coprime) integers $A$ and $B$. When $L\in G_{(A,B)}$, the procedure also yields a complete, explicit description of all admissible exponents $K$. The method begins by extracting the common factor $g=\gcd(A,B)$ and decomposing $L$ into its $g$-part $\lambda_{\mathcal P(g)}(L)$ and coprime core $\ell$. From there, the decision   follows from the earlier characterizations: the threshold $\gamma(L)$ accounts for the entire $g$-contribution, while the goodness of   $\ell$ is   determined form the classical coprime case. 

 \medskip

\begin{center}
\begin{tabular}{p{14cm}}
\hline 
\noindent\textbf{Input:} Non-Zero Integers $A$, $B$,  and $L\ge 1$.

\noindent\textbf{Output:} (i) Feasibility: whether $L\in G_{(A,B)}$; 
(ii) the minimal admissible exponent $\min \mathcal K_{(A,B)}(L)$  (if feasible);
(iii) the full set $\mathcal K_{(A,B)}(L)$ of admissible exponents (if feasible).

\medskip
\noindent\textbf{Step 0 (trivial case).}
If $L=1$, return \textsc{Yes}, $\min \mathcal K_{(A,B)}(1) =1$, and $\mathcal K_{(A,B)}(1)=\mathbb{N}$.

\medskip
\noindent\textbf{Step 1 (split $L$).}
Compute $g:=\gcd(A,B)$ and write $A=ga$, $B=gb$ with $\gcd(a,b)=1$.
Let $\mathcal P(g)$ be the set of prime divisors of $g$ and decompose
\[
   L=\lambda_{\mathcal P(g)}(L)\ell,\quad
   \lambda_{\mathcal P(g)}(L)=\prod_{p\mid g}p^{\nu_p(L)},\quad \text{and} \quad 
   \ell=\prod_{p\nmid g}p^{\nu_p(L)}.
\]
Compute the threshold
\[
   \gamma(L):=\max_{p\mid g}\left\lceil\frac{\nu_p(L)}{\nu_p(g)}\right\rceil
   \quad(\text{take }\gamma(L)=0 \text{ if } g=1).
\]

\medskip
\noindent\textbf{Step 2 ($g$-part).}
If $\ell=1$, then $L\in G_{(A,B)}$ and
\[
  \mathcal K_{(A,B)}(L)=\{K\in\mathbb{N}:K\ge\gamma(L)\},\quad
  K_{\min}=\max\{1,\gamma(L)\}.
\]
Return \textsc{Yes}.

\medskip
\noindent\textbf{Step 3 (necessity for $\ell$).}
If $\gcd(\ell,a)\neq 1$ or $\gcd(\ell,b)\neq 1$, return \textsc{No}.
(Otherwise the quantity ${\rm ord}_\ell (ab^{-1})  $ is well-defined.)

\medskip
\noindent\textbf{Step 4 (the feasibility $\ell\in G_{(a,b)}$).}
 Decide $\ell\in G_{(a,b)}$) via Theorem~\ref{goodP}  and Propositions~\ref{evengood}--\ref{prop2m}. (Alternatively,  apply \cite[Algorithm 1]{J2025}.)

\medskip
\noindent\textbf{Step 5 (the admissible set  $\mathcal K_{(A,B)}(L)$ and $\min \mathcal K_{(A,B)}(L)$).}

\begin{enumerate}[(i)]
    \item Define
\[
  L_0=\begin{cases}
         1,& \ell\in\{1,2\},\\[2pt]
         {\rm ord}_{\ell}(ab^{-1}),& \ell\ge 3,
      \end{cases}
  \quad
  r=\begin{cases}
        0,& \ell\in\{1,2\},\\[2pt]
        \dfrac{L_0}{2},& \ell\ge 3.
     \end{cases}
\]
(For $\ell\ge 3$, $L_0$ is even and $r$ is an integer.)

\item By Proposition~\ref{prop:AP-noncoprime},
\[
  \mathcal K_{(A,B)}(L)
   = \{K\in\mathbb{N}:\ K\equiv r\pmod{L_0}\ \text{and}\ K\ge \gamma(L)\},
\]
which is infinite. 
\item The minimal admissible exponent is
\[
  K_{\min}=\gamma(L)+\bigl((r-\gamma(L))\bmod L_0\bigr).
\]
\end{enumerate}
Return \textsc{Yes}, $K_{\min}$, and the AP description above. \\
\hline
\end{tabular}
\end{center}

 The following examples demonstrate the algorithm’s application, detailing key computations, boundary cases, and the resulting structure of admissible exponents.

\begin{example} 
Let $A=18$, $B=12$, and $L=72$. Then $g=\gcd(18,12)=6$ with $\mathcal P(g)=\{2,3\}$.
Then  $L=2^3\cdot 3^2$ which implies that  
\[
\lambda_{\mathcal P(g)}(L)=2^3\cdot 3^2=72,\quad \ell=1,\quad\text{and }~
\gamma(L)=\max\{\lceil 3/1\rceil,\lceil 2/1\rceil\}=3.
\]
Since $\ell=1$,  we have   $\ell\in G_{(3,2)}$ which implies that  $L\in G_{(A,B)}$ by Theorem~\ref{thm:equiv-G-gamma}.  Hence, 
\[
\mathcal K_{(A,B)}(L)=\{K\in\mathbb{N}:K\ge 3\}\quad \text{and} \quad \min \mathcal K_{(A,B)}(L) =3
\]
by   Proposition~\ref{prop:AP-noncoprime}.
\end{example}

\begin{example} 
Let $A=10$, $B=15$, and $L=6$. Then $g=\gcd(10,15)=5$ which implies that  $A=ga$ and $B=gb$ with $a=3$ and $b=2$.
Since $\mathcal P(g)=\{5\}$ and $L=2\cdot 3$, we have
\[
\lambda_{\mathcal P(g)}(L)=1,\quad \ell=6,\quad \gcd(\ell,a)=\gcd(6,3)=3\neq 1.
\]
Step~3   returns \textsc{No}. Thus $L\notin G_{(A,B)}$.
\end{example}

\begin{example} 
Let $A=6$, $B=3$, and $L=15$. Then $g=\gcd(6,3)=3$, $a=2$, $b=1$, and $\mathcal P(g)=\{3\}$.
Decompose $L=3\cdot 5$ which implies that $\lambda_{\mathcal P(g)}(L)=3$, $\ell=5$, and
\[
\gamma(L)=\left\lceil \frac{\nu_3(L)}{\nu_3(g)}\right\rceil=\lceil 1/1\rceil=1.
\]
For the core $\ell$-part, ${\rm ord}_{5}(ab^{-1})={\rm ord}_{5}(2)=4$ which means $\ell\in G_{(2,1)}$,
\[
L_0={\rm ord}_{\ell}(ab^{-1})=4\quad\text{and} \quad  r=\frac{L_0}{2}=2.
\]
By Proposition~\ref{prop:AP-noncoprime},
\[
\mathcal K_{(A,B)}(L)=\{K\in\mathbb{N}:\ K\equiv 2\pmod 4,\ K\ge 1\}=\{2,6,10,\dots\},
\quad\text{and} \quad \min \mathcal K_{(A,B)}(L) =2.
\]
\end{example}

\section{Conclusion and Remarks}
\label{sec:4}
The classical theory of good integers has been extended and completed by treating the non-coprime setting in a concise and  self-contained way. The characterization has been presented   in terms of   the classical coprime criterion   and $p$-adic valuations of $\ell$. As a consequence, the set of admissible exponents has been described as a single arithmetic progression, truncated below by the threshold, and an explicit formula for the minimal admissible  exponent has been provided. Special cases   have also been included.
Beyond the structural results, we present an implementable decision procedure: it decomposes $A,B,L$, determines the threshold $\gamma(L)$, analyzes the coprime core, and outputs both the decision and the full arithmetic progression of admissible exponents.

For further study, it would be interesting to investigate the asymptotic behavior and density of good integers in this general setting, as well as to explore potential applications.

 \section*{Acknowledgments}
This research was funded by the National Research Council of Thailand and Silpakorn University  under Research Grant  N42A650381.


\end{document}